\documentclass{birkjour}
\newtheorem{Theorem}{Theorem}[section]
\newtheorem{Corollary}{Corollary}[section]
\newtheorem{Definition}{Definition}[section]
\newtheorem*{Definition*}{Definition}
\newtheorem{Example}{Example}[section]

\newtheorem{Proposition}{Proposition}[section]

\newtheorem{Remark}{Remark}[section]

\numberwithin{equation}{section}
\renewcommand{\proofname}{proof}

\usepackage{color}                    
\begin{document}
	
	%
	%
	%
	%
	%
	%
	%
	%
	%

	\title[The skew James type constant in Banach spaces]
	{The skew James type constant\\in Banach spaces}

	\author[Zhiyong Rao]{Zhiyong Rao}
	\address{School of Mathematics and Physics, Anqing Normal University, Anqing 246133, People’s Republic of China}
	\email{17855677860@163.com}
	\thanks{$^{*}$Corresponding author}
	\author{Qi Liu$^{*}$}
	\address{School of Mathematics and Physics, Anqing Normal University, Anqing 246133, People’s Republic of China}
	\email{liuq325@mail2.sysu.edu.cn}
	\author{Qiong Wu}
	\address{School of Mathematics and Physics, Anqing Normal University, Anqing 246133, People’s Republic of China}
	\email{wuqiongaqtc@163.com}
	\author{Zhouping Yin}
	\address{School of Mathematics and Physics, Anqing Normal University, Anqing 246133, People’s Republic of China}
	\email{yzp@aqnu.edu.cn}
	\author{Qichuan Ni}
	\address{School of Mathematics and Physics, Anqing Normal University, Anqing 246133, People’s Republic of China}
	\email{080821074@stu.aqnu.edu.cn}
	
	\subjclass{46B20}
	
	\keywords{Banach Spaces; James type constants; von Neumann-Jordan type constants}
	
	\date{---}
	\dedicatory{---}
	
	\begin{abstract}
		In the past, Takahashi has introduced the James type constants $\mathcal{J}_{\mathcal{X}, t}(\tau)$. Building upon this foundation, we introduce an innovative skew James type constant, denoted as $\mathcal{J}_{ t}[\tau,\mathcal{X}]$, which is perceived as a skewed counterpart to the traditional James type constants. We delineate a novel constant, and proceed to ascertain its equivalent representations along with certain attributes within the context of Banach spaces, and then an investigation into the interrelation between the skewness parameter and the modulus of convexity is conducted, after that we define another new constant $\mathcal{G}_t(\mathcal{X}),$ and some conclusions were drawn.
	\end{abstract}
	
	\maketitle
	\section{Introduction}
	Functional analysis is a branch developed from modern mathematics. Since the beginning of the 20th century, it has gradually developed into a quite complete scientific theory through the research of mathematicians. In various research fields of functional analysis, geometric theory in Banach Spaces not only provides basic tools, but also provides important research results, such as fixed point theory and approximation theory. Among them, the theory of fixed points is the most important application of Banach space geometry theory, which is self-evident in importance and is widely used in many different fields.\\
	Since Banach spaces are a broad class of abstract spaces, it is quite challenging to describe their geometric structure comprehensively and intuitively. Therefore, one can consider using spatial geometric constants to study the geometric structure of space. Since spatial geometric constants quantify the properties of spatial geometry, using these constants can transform the problem of describing abstract geometric structures into a computational issue involving specific functions, thus advancing from qualitative discussion to quantitative calculation. So in this theory, the geometric properties of Banach Spaces usually play an important role, such as uniform non-squareness, uniform convexity, normal structure, and so on.\\
	The quantization of geometric constants describing geometric features in Banach space is very important for the study of geometric properties, this has also attracted the interest of many scholars \cite{2, 5, 13, 14, 19, 20, 21, 22, 26, 27}. In recent years, geometric constants describing the geometric properties of Banach spaces have been extensively defined and studied, which makes some problems in Banach spaces easier to handle and solve. The geometric constants of a specific space help deepen our understanding of its geometric features and provide robust support for further research endeavors.\\
	In 1936, Clarkson \cite{10} introduced the modulus of convexity $\delta_\mathcal{X}\left(\varepsilon\right),$ which is an effective tool for studying regular structures. It is the basis of the method used by Jordan and von Neumann to characterize inner product spaces.\\
	On this basis, in 1937, Clakson defined the von Neumann-Jordan type constants \cite{4}, that is\\
	$$\mathcal{C}_{\mathcal{N}\mathcal{J}}(\mathcal{X})=\sup\left\{\frac{\left\|x_1+x_2\right\|^2+\left\|x_1-x_2\right\|^2}{2\left\|x_1\right\|^2+\left\|x_2\right\|^2}:x_1,x_2\in \mathcal{X},\left\|x_1\right\|^2+\left\|x_2\right\|^2\neq0\right\}\cdot$$
	\\In 1964, James introduced the James constant\\
	$$\mathcal{J}(\mathcal{X})=\sup\left\{\min\left\{\left\|x_1+x_2\right\|,\left\|x_1-x_2\right\|\right\}{:}x_1,x_2\in \mathcal{B}_{\mathcal{X}}\right\},$$
	\\it is used to describe the uniform non-square, so it is also called the non-square constant \cite{24}.\\
	In 1970, Goebelx \cite{23} proposed the convex coefficient\\ $$\varepsilon_0\left(\mathcal{X}\right)=\sup\left\{\varepsilon>0:\delta_\mathcal{X}\left(\varepsilon\right)=0\right\},$$
	\\and use it to study consistent formal structures.\\
	In 1994, J.Garcia-Falset introduced the constant $\mathcal{R}\left(\mathcal{X}\right),$ that is\\
	$$\mathcal{R}\left(\mathcal{X}\right)=\sup\left\{\underset{n\to\infty}{\operatorname*{\operatorname*{limsup}}}\Vert x_{n}+x\Vert{:}x,x_{n}\in \mathcal{B}_{\mathcal{X}},n=1,2,\cdots,x_{n}\xrightarrow{\omega}0\right\},$$
	\\and he proved in 1997 that Banach Spaces have a fixed point property when $\mathcal{R}(\mathcal{X})<2$ \cite{25}.\\
	In 2002, Zbaganu proposed a new constant called the Zbaganu constant \cite{2,14,15}. It is expressed in the following form\\
	$$\mathcal{C}_\mathcal{Z}(\mathcal{X})=\sup\left\{\frac{\parallel x_1+x_2\parallel\parallel x_1-x_2\parallel}{\parallel x_1\parallel^2+\parallel x_2\parallel^2}{:}x_1,x_2\in \mathcal{X},(x_1,x_2)\neq(0,0)\right\}.$$
	\\In 2008, Yang C. S. et al. gave some properties of the generalized James constant $\mathcal{J}[t, \mathcal{X}],$ and demonstrated that when $t\geq1$, if $\mathcal{J}\left(t,\mathcal{X}\right)<\frac{t+\sqrt{t^{2}+4}}{2},$ then Banach space $\mathcal{X}$ has a uniformly normal structure.\\
	In 2010, Moslehian and Rassias succeeded in deriving an innovative equivalent formulation of the inner product space as documented in the scholarly literature \cite{29}.
	For any nonnegative real numbers $\lambda,\mu$ and any $x_1,x_2\in \mathcal{X}$, we define that:\\
	A normed space $(\mathcal{X},\|\cdot\|)$ is an inner product space if and only if\\ $$\|\lambda x_1+\mu x_2\|^2+\|\mu x_1-\lambda x_2\|^2=(\lambda^2+\mu^2)(\|x_1\|^2+\|x_2\|^2).$$
	\\In 2012, Yang C. S. et al. \cite{20} studied and expanded some properties of the von Neumann-Jordan constant, used the von Neumann-Jordan constant to describe the equivalence conditions of uniform non-square, and introduced the relationship between $\mathcal{C}_{1}^{^{\prime}}(\mathcal{X})$ and other spatial constants, such as the relationship between $\mathcal{C}_{1}^{^{\prime}}(\mathcal{X})$ and $\mathcal{C}_1(\mathcal{X})$.\\
	In 2019, A. mini-Harandi and M.Rahimi \cite{28} introduced some new constants $\mathcal{C}_{\lambda,p}(\mathcal{X})$ and $\mathcal{C}_{\lambda,p}^{^{\prime}}(\mathcal{X})$ of real normed space $\mathcal{X}$ based on the research work of many scholars, and provided a series of results about these constants. Then, they give some descriptions of Hilbert Spaces and uniformly non-square Spaces, and finally get the conditions under which Banach space $\mathcal{X}$ contains a normal structure.\\
	In 2021, Liu et al. introduced the constant $\mathcal{L}^{\prime}_{\mathcal{Y}\mathcal{J}}(\lambda,\mu,\mathcal{X})$, as follows\\
	$$\mathcal{L}^{\prime}_{\mathcal{Y}\mathcal{J}}(\lambda,\mu,\mathcal{X})=\sup\left\{\frac{\|\lambda x_1+\mu x_2\|^2+\|\mu x_1-\lambda x_2\|^2}{2(\lambda^2+\mu^2)},x_1,x_2\in \mathcal{S}_\mathcal{X} \right\}.$$
	\\We collected some of its properties in Banach space \cite{1}.\\
	\\(i) $1 \leq \mathcal{L}_{\mathcal{Y} \mathcal{J}}^{\prime}(\lambda, \mu, \mathcal{X}) \leq 1+\frac{2 \lambda \mu}{\lambda^2+\mu^2}$;\\
	(ii) $\mathcal{X}$ is Hilbert space when and only when $\mathcal{L}_{\mathcal{Y} \mathcal{J}}^{\prime}(\lambda, \mu, \mathcal{X})=1$;\\
	(iii) $\mathcal{X}$ is uniformly non-square when and only when\\ 
	$$\mathcal{L}_{\mathcal{Y} \mathcal{J}}^{\prime}(\lambda, \mu, \mathcal{X})<1+\frac{2 \lambda \mu}{\lambda^2+\mu^2}.$$
	\\Particularly, let $\lambda=1$ and $\mu=\tau$, for the new constant $\mathcal{J}_{ t}[\tau,\mathcal{X}]$ introduced in this paper, if $t=2$, we have\\
	$$\mathcal{L}^{\prime}_{\mathcal{Y}\mathcal{J}}(\lambda,\mu,\mathcal{X})=\frac{\left(\mathcal{J}_t[\tau, \mathcal{X}]\right)^2}{1+\tau^2}.$$
	\\Two years later, Yang et al. calculated the exact value of the constant\\ $\mathcal{L}_{\mathcal{Y} \mathcal{J}}^{\prime}(\lambda, \mu, \mathcal{X})$ in the regular octagon space \cite{8}.\\
	In 2023, Fu et al. introduced three constants: $\mathcal{E}[t, \mathcal{X}], \mathcal{C}_{\mathcal{N}\mathcal{J}}[\mathcal{X}],$ and $\mathcal{J}[t, \mathcal{X}]$. They are viewed as skew versions of the Gao’s parameter, the Von Neumann-Jordan constant, and the Milman’s moduli, respectively. They are defined as follows\\
	$$\begin{aligned}&\mathcal{E}[t, \mathcal{X}]=\sup\{\|x_1+tx_2\|^2+\|tx_1-x_2\|^2:x_1,x_2\in \mathcal{S}_\mathcal{X}\}, \quad t\in\mathbb{R},\\
		&\mathcal{C}_{\mathcal{N}\mathcal{J}}[\mathcal{X}]=\sup\left\{\frac{\|x_1+tx_2\|^{2}+\|tx_1-x_2\|^{2}}{2(1+t^{2})}:x_1,x_2\in\mathcal{S}_\mathcal{X}\right\},0\leq t\leq1,\\
		&\mathcal{J}[t, \mathcal{X}]=\sup\{\min\{\|x_1+tx_2\|,\|tx_1-x_2\|\}:x_1,x_2\in \mathcal{S}_\mathcal{X}\},\quad t>0.\end{aligned}$$
	\\For more information about them, please refer to \cite{9}.
	
	\section{Notions and preliminaries}
	In this section, we present some definitions and preliminary results that will be used throughout this article.\\
	Throughout this article, we assume that $\mathcal{X}$ is a non-trivial Banach space and that its dimension is at least 2. $\mathcal{B}_\mathcal{X} =\{x \in \mathcal{X} : \|x\| \leq 1\}$ denotes the unit ball of the Banach space $\mathcal{X}$ and $\mathcal{S}_\mathcal{X} = \{x \in \mathcal{X} :\|x\| = 1\}$ be the unit sphere of $\mathcal{X}$.\\
	As a generalization of the James constant, Takahashi introduced the James type constant $\mathcal{J}_{\mathcal{X}, t}(\tau)$, as follows:
	\begin{Definition} \cite{2} For $\tau\geq0$ and $-\infty \leq t < +\infty$, the constant $\mathcal{J}_{\mathcal{X}, t}(\tau)$ is defined to be\\
		$$\mathcal{J}_{\mathcal{X}, t}(\tau) = \sup \{\mathcal{M}_t(\| x_1 +\tau x_2\| , \| x_1 -\tau x_2\| ) : x_1, x_2\in \mathcal{S}_\mathcal{X} \}.$$
		\\Particularly, for the new constant $\mathcal{J}_{t}[\tau,\mathcal{X}]$ introduced in this paper, let $\tau=1$, we have\\
		$$\mathcal{J}_{ t}[\tau,\mathcal{X}]=\mathcal{J}_{\mathcal{X}, t}(\tau).$$
		\\Here, $\mathcal{M}_t(a,b)$ is the generalized average of $a$ and $b$ $\in (0, \infty)$, and $\mathcal{M}_t(a,b)$ is different with different values of $t$. \\if $t \in \mathbb {R} \setminus \{0\}$, as follows\\
		$$\mathcal{M}_t(a,b)= \left(\frac{a^t+b^t}{2}\right) ^{\frac{1}{t}}.$$
		\\As we all know, $\mathcal{M}_t(a,b)$ is nondecreasing. In particularly,\\
		$$\mathcal{M}_+\infty(a,b)=\lim_{t \rightarrow +\infty} \mathcal{M}_t(a, b) = \max\{a,b\}.$$
		$$\mathcal{M}_-\infty(a,b)=\lim_{t \rightarrow -\infty} \mathcal{M}_t(a, b) = \min\{a,b\}.$$
		$$\mathcal{M}_0(a,b)=\lim_{t \rightarrow 0} \mathcal{M}_t(a, b) = \sqrt {ab} \text{ }.$$
	\end{Definition}
	Then, as a generalization of the von Neumann-Jordan constant, Takahashi also introduced the von Neumann-Jordan type constant $\mathcal{C}_t(\mathcal{X})$, that is:
	\begin{Definition}
		\cite{2} For $\tau\geq0$ and $-\infty \leq t < +\infty$, the constant $\mathcal{C}_t(\mathcal{X})$ is defined to be\\
		$$\mathcal{C}_t(\mathcal{X})=\sup\left\{\frac{\mathcal{J}_{\mathcal{X}, t}(\tau)^2}{1+\tau^2}:0\leq \tau\leq1\right\}.$$
	\end{Definition}
	Obviously, these theorems contain some known constants, for instance, the James constant $\mathcal{J}(\mathcal{X})=\mathcal{J}_{\mathcal{X}, -\infty}(1)$, the $\mathcal{A}_2$-constant $\mathcal{A}_2(\mathcal{X}) =\mathcal{J}_{\mathcal{X}, 1}(1)$ \cite{7}, the Gao constant  $\mathcal{E}(\mathcal{X})=2\mathcal{J}_{\mathcal{X},2}(1)^{2}$\cite{6}, the Alonso-Llorens-Fuster constant $\mathcal{T}(\mathcal{X}) =\mathcal{J}_{\mathcal{X}, 0}(1)$ \cite{5}, the von Neumann-Jordan constant $\mathcal{C}_{\mathcal{N}\mathcal{J}}(\mathcal{X})=\mathcal{C}_2(\mathcal{X})$, and the $\mathrm{Zbaganu}$ constants $\mathcal{C}_\mathbf{\mathcal{Z}}(\mathcal{X})=\mathcal{C}_0(\mathcal{X}).$ In particular, by taking $t=-\infty$, the constant $\mathcal{C}_{-\infty}(\mathcal{X})$ is obtained, some of whose related results are described in\cite{2, 20, 22}.
	
	\begin{Definition}
		Let $x_1,x_2\in \mathcal{S}_{\mathcal{X}},$ if there exists $\delta>0,$ such that\\ 
		$$\frac{\|x_1+x_2\|}{2}\leq1-\delta \hspace*{0.5em} or \hspace*{0.5em}
		\frac{\|x_1-x_2\|}{2}\leq1-\delta,$$
		\\then the space $\mathcal{X}$ is called Uniformly non-square.
	\end{Definition}

	\section{Main Results}
	First, we give the following results.\\
	We define a new skew constant $\mathcal{J}_{t}[\tau,\mathcal{X}]$ based on the James type constants $\mathcal{J}_{\mathcal{X}, t}(\tau),$ which covers the skew constants $\mathcal{E}[t, \mathcal{X}]$ and $\mathcal{J}[t, \mathcal{X}]$ for $t>0.$

	
	\begin{Definition}
		
		Let $\tau\geq0$ and $-\infty \leq t<\infty$, the skew James type constant $\mathcal{J}_{t}[\tau,\mathcal{X}]$ is defined as\\
		$$
		\left\{\begin{array}{l}
			\sup \left\{\left(\frac{\|x_1+\tau x_2\|^t+\|\tau x_1- x_2\|^t}{2}\right)^{\frac{1}{t}}:\|x_1\|=1,\|x_2\|=1\right\}, \quad-\infty<t \neq 0<\infty, \\
			\sup \{\sqrt{\| x_1+\tau x_2\|\|\tau x_1- x_2\|}:\|x_1\|=1,\|x_2\|=1\}, \quad t=0, \\
			\sup \{\min \{\| x_1+ \tau x_2\|,\|\tau x_1- x_2\|\}:\|x_1\|=1,\|x_2\|=1\}, \quad t=-\infty.
		\end{array}\right.
		$$
	\end{Definition}
	By using the generalized mean, we have
	\begin{Definition}  Let $\tau \geq 0$ and  $-\infty \leq t<\infty$, the skew James type constant can be defined as\\
		$$
		\mathcal{J}_{ t}[\tau,\mathcal{X}]=\sup \left\{\mathcal{\mathcal{M}}_t(\| x_1+ \tau x_2\|,\|\tau x_1-x_2\|): x_1, x_2 \in \mathcal{S}_\mathcal{X}\right\}.
		$$
	\end{Definition}
	
	\begin{Remark} According to the definition of the constant $\mathcal{J}_{t}[\tau,\mathcal{X}]$, we can easily get the following facts:\\
		\\(i)	$\mathcal{J}_{-\infty}[\tau,\mathcal{X}]=\mathcal{J}[t, \mathcal{X}]$ where $\tau, t>0$.\\
		(ii) $2(\mathcal{J}_{2}[\tau,\mathcal{X}])^2=\mathcal{E}[t, \mathcal{X}]$ where $\tau, t\geq0$.\\
	\end{Remark}
	
	\begin{Proposition} Let $\mathcal{X}$ be a Banach space. The following two statements are true:\\
		\\(i) When $\tau \in [0,\infty],$ $\left(\mathcal{J}_t[\tau, \mathcal{X}]\right)^{t}$ is a convex function.\\
		(ii) When $\tau \in [0,\infty],$ $\left(\mathcal{J}_t[\tau, \mathcal{X}]\right)^{t}$ is a continuous function.\\
		
	\end{Proposition}
	\begin{proof} 
		(i) Let $\tau_1, \tau_2 \in [0,\infty], \lambda \in(0,1)$. Then, for any $x_1, x_2 \in \mathcal{S}_\mathcal{X}$, we can derive that\\
		$$
		\begin{aligned}
			& \frac{\|x_1+\left(\lambda \tau_1+(1-\lambda) \tau_2\right)x_2\|^t+\|\left(\lambda \tau_1+(1-\lambda)\tau_2\right) x_1-x_2\|^t}{2}\\
			&\leq \frac{ \left(\lambda\left\|x_1+\tau_1 x_2\right\|+(1-\lambda)\left\|x_1+\tau_2 x_2\right\|\right)^t}{2} \\
			&+\frac{\left(\lambda\left\|\tau_1 x_1-x_2\right\|+(1-\lambda)\left\|\tau_2 x_1-x_2\right\|\right)^t}{2} \\
			&\leq \frac{\lambda\left(\left\|x_1+\tau_1 x_2\right\|^t+\left\|\tau_1 x_1-x_2\right\|^t\right)}{2}+\frac{(1-\lambda)\left(\left\|x_1+\tau_2 x_2\right\|^t+\left\|\tau_2 x_1-x_2\right\|^t\right)}{2} \\
			&\leq \lambda \mathcal{J}_t^{t} [\tau_1, \mathcal{X}]+(1-\lambda) \mathcal{J}_t^{t} [\tau_2, \mathcal{X}],
		\end{aligned}
		$$
		\\which implies that\\
		$$
		\left(\mathcal{J}_t\left[\lambda \tau_1+(1-\lambda) \tau_2, \mathcal{X}\right]\right)^{t} \leq \lambda \left(\mathcal{J}_t[\tau_1, \mathcal{X}]\right)^{t}+(1-\lambda) \left(\mathcal{J}_t[\tau_2, \mathcal{X}]\right)^{t} .
		$$
		\\(ii) By (i), we can obtain $\left(\mathcal{J}_t[\tau, \mathcal{X}]\right)^{t}$ is continuous on $\mathbb{R}$.
	\end{proof}
	Subsequently‌, we give the equivalent forms of the constant $\mathcal{J}_t [\tau, \mathcal{X}]$ for $t\geq1$.\\
	\begin{Proposition}\label{1}
		Let $\mathcal{X}$ be a Banach space and $t\geq1$. Then\\
		\\(i) $\mathcal{J}_t [\tau, \mathcal{X}]=\sup \left\{\left(\frac{\|x_1+\tau x_2\|^t+\|\tau x_1- x_2\|^t}{2}\right)^{\frac{1}{t}}:x_1,x_2\in \mathcal{B}_\mathcal{X}\right\}.$\\
		(ii) If $\mathcal{X}$ is a reflexive Banach space, then\\
		$$
		\mathcal{J}_t [\tau, \mathcal{X}]=\sup \left\{\left(\frac{\|x_1+\tau x_2\|^t+\|\tau x_1- x_2\|^t}{2}\right)^{\frac{1}{t}}:x_1,x_2\in \operatorname{ext}\left(\mathcal{B}_\mathcal{X}\right)\right\}.
		$$
		\\(iii) $\mathcal{J}_t [\tau, \mathcal{X}]=\sup \left\{\mathcal{J}_t [\tau, \mathcal{X}_{0}]: \mathcal{X}_0 \subset \mathcal{X}, \operatorname{dim}\left(\mathcal{X}_0\right)=2\right\}$.
	\end{Proposition}
	\begin{proof} (i) Actually, we just have to prove that\\ 
		$$\begin{aligned}&\sup \left\{\|x_1+\tau x_2\|^t+\|\tau x_1- x_2\|^t:x_1,x_2\in \mathcal{B}_\mathcal{X}\right\}\\
			&= \sup \left\{\|x_1+\tau x_2\|^t+\|\tau x_1- x_2\|^t:x_1,x_2\in \mathcal{S}_\mathcal{X}\right\}.\end{aligned}$$
		\\First, for any $x_1, x_2 \in \mathcal{B}_\mathcal{X}$, we have\\
		$$
		\begin{aligned}
			& x_1=\frac{1-\|x_1\|}{2}\left(-\frac{x_1}{\|x_1\|}\right)+\left(1-\frac{1-\|x_1\|}{2}\right) \frac{x_1}{\|x_1\|}, \\
			& x_2=\frac{1-\|x_2\|}{2}\left(-\frac{x_2}{\|x_2\|}\right)+\left(1-\frac{1-\|x_2\|}{2}\right) \frac{x_2}{\|x_2\|} .
		\end{aligned}
		$$
		\\As a matter of convenience, we denote $\frac{1-\|x_1\|}{2}$, $1-\frac{1-\|x_1\|}{2}$, $\frac{1-\|x_2\|}{2}$, $1-\frac{1-\|x_2\|}{2}$, $-\frac{x_1}{\|x_1\|}$, $\frac{x_1}{\|x_1\|}$, $-\frac{x_2}{\|x_2\|}$, $\frac{x_2}{\|x_2\|}$ by $m_1, m_2, n_1, n_2, x_{11}, x_{12}, x_{21}, x_{22}$, respectively. Thus\\
		$$x_1=m_1x_{11}+m_2x_{12},x_2=n_1x_{21}+n_2x_{22}.$$
		\\It is obvious that $x_{11},x_{12},x_{21},x_{22}\in \mathcal{S}_\mathcal{X}$ and $m_1,m_2,n_1,n_2\in [0,1]$ such that $m_1+m_2=1$ and $n_1+n_2=1$.
		Let $t\geq1$. Since $f(x)=x^t$ is a convex function on $[0, \infty)$, for any $x_1, x_2 \in \mathcal{B}_\mathcal{X}$, we have\\
		$$
		\begin{aligned}
			&\|x_1+\tau x_2\|^t+\|\tau x_1- x_2\|^t\\
			&=\left\|\sum_{\phi=1}^2 m_\phi x_{1\phi}+\tau\left(\sum_{\phi=1}^2 m_\phi x_2\right)\right\|^t+\left\|\tau\left(\sum_{\phi=1}^2 m_\phi x_{1\phi}\right)-\left(\sum_{\phi=1}^2 m_\phi x_2\right)\right\|^t \\
			&=\left\|\sum_{\phi=1}^2 m_\phi\left(x_{1\phi}+\tau x_2\right)\right\|^t+\left\|\sum_{\phi=1}^2 m_\phi\left(\tau x_{1\phi}-x_2\right)\right\|^t \\
			& \leq\left(\sum_{\phi=1}^2 m_\phi\left\|x_{1\phi}+\tau x_2\right\|\right)^t+\left(\sum_{\phi=1}^2 m_\phi\left\|\tau x_{1\phi}-x_2\right\|\right)^t \\
			& \leq \sum_{\phi=1}^2 m_\phi\left\|x_{1\phi}+\tau x_2\right\|^t+\sum_{\phi=1}^2 m_\phi\left\|\tau x_{1\phi}-x_2\right\|^t \\
			&=\sum_{\phi=1}^2 m_\phi\left\|\sum_{\psi=1}^2 n_\psi x_{1\phi}+\tau\left(\sum_{\psi=1}^2 n_\psi x_{2\psi}\right)\right\|^t\\
			&+\sum_{\phi=1}^2 m_\phi\left\|\tau\left(\sum_{\psi=1}^2 n_\psi x_{1\phi}\right)-\left(\sum_{\psi=1}^2 n_\psi x_{2\psi}\right)\right\|^t \\
		\end{aligned}$$
		$$\begin{aligned}
			&=\sum_{\phi=1}^2 m_\phi\left\|\sum_{\psi=1}^2 n_\psi\left(x_{1\phi}+\tau x_{2\psi}\right)\right\|^t+\sum_{\phi=1}^2 m_\phi\left\|\sum_{\psi=1}^2 n_\psi\left(\tau x_{1\phi}-x_{2\psi}\right)\right\|^t \\
			& \leq \sum_{\phi=1}^2 m_\phi\left(\sum_{\psi=1}^2 n_\psi\left\|x_{1\phi}+\tau x_{2\psi}\right\|\right)^t+\sum_{\phi=1}^2 m_\phi\left(\sum_{\psi=1}^2 n_\psi\left\|\tau x_{1\phi}-x_{2\psi}\right\|\right)^t\\
			& \leq \sum_{\phi=1}^2 m_\phi\sum_{\psi=1}^2 n_\psi\left\|x_{1\phi}+\tau x_{2\psi}\right\|^t+ \sum_{\phi=1}^2 m_\phi\sum_{\psi=1}^2 n_\psi\left\|\tau x_{1\phi}-x_{2\psi}\right\|^t\\
			&= \sum_{\phi=1}^2 m_\phi\sum_{\psi=1}^2 n_\psi\left(\left\|x_{1\phi}+\tau x_{2\psi}\right\|^t+\left\|\tau x_{1\phi}-x_{2\psi}\right\|^t\right)\\
			&\leq\max\{\left\|x_{1\phi}+\tau x_{2\psi}\right\|^t+\left\|\tau x_{1\phi}-x_{2\psi}\right\|^t:\phi=1,2,\psi=1,2\}\\
			&\leq\sup\{\left\|x_1+\tau x_2\right\|^t+\left\|\tau x_1-x_2\right\|^t:x_1,x_2\in \mathcal{S}_\mathcal{X}\}.
		\end{aligned}
		$$
		\\This shows that\\
		$$\mathcal{J}_t [\tau, \mathcal{X}]\geq\sup \left\{\left(\frac{\|x_1+\tau x_2\|^t+\|\tau x_1- x_2\|^t}{2}\right)^{\frac{1}{t}}:x_1,x_2\in \mathcal{B}_\mathcal{X}\right\}.$$
		\\In addition, it is obvious that\\
		$$\mathcal{J}_t [\tau, \mathcal{X}]\leq\sup \left\{\left(\frac{\|x_1+\tau x_2\|^t+\|\tau x_1- x_2\|^t}{2}\right)^{\frac{1}{t}}:x_1,x_2\in \mathcal{B}_\mathcal{X}\right\}.$$
		\\(ii) Again, we just have to prove that\\
		$$\begin{aligned}&\sup \left\{\|x_1+\tau x_2\|^t+\|\tau x_1- x_2\|^t: x_1,x_2\in \operatorname{ext}\left(\mathcal{B}_\mathcal{X}\right)\right\}\\
			&=\sup \left\{\|x_1+\tau x_2\|^t+\|\tau x_1- x_2\|^t:x_1,x_2\in \mathcal{S}_\mathcal{X}\right\}.
		\end{aligned}$$
		\\Since $\mathcal{X}$ is a reflexive Banach space, then, according to Krein-Milman theorem, we know that $\mathcal{B}_\mathcal{X}=\overline{\operatorname{co}}\left(\operatorname{ext}\left(\mathcal{B}_\mathcal{X}\right)\right)$. Then, by (i), we obtain\\
		$$
		\mathcal{J}_t [\tau, \mathcal{X}]=\sup \left\{\left(\frac{\|x_1+\tau x_2\|^t+\|\tau x_1- x_2\|^t}{2}\right)^{\frac{1}{t}}: x_1, x_2 \in \overline{\operatorname{co}}\left(\operatorname{ext}\left(\mathcal{B}_\mathcal{X}\right)\right)\right\}.
		$$
		\\Further, according to the continuity of norm, it is easy to know\\
		$$
		\mathcal{J}_t [\tau, \mathcal{X}]=\sup \left\{\left(\frac{\|x_1+\tau x_2\|^t+\|\tau x_1- x_2\|^t}{2}\right)^{\frac{1}{t}}: x_1, x_2 \in \operatorname{co}\left(\operatorname{ext}\left(\mathcal{B}_\mathcal{X}\right)\right)\right\}.
		$$
		\\Next, for any $x_1, x_2 \in \operatorname{co}\left(\operatorname{ext}\left(\mathcal{B}_\mathcal{X}\right)\right)$, there exist $\left\{x_{1\phi}\right\}_{\phi=1}^{n_{x_1}},\left\{x_{2\psi}\right\}_{\psi=1}^{n_{x_2}} \in \operatorname{ext}\left(\mathcal{B}_\mathcal{X}\right)$ and $\left\{m_\phi\right\}_{\phi=1}^{n_{x_1}},\left\{n_\psi\right\}_{\psi=1}^{n_{x_2}} \in[0,1]$ such that\\
		$$
		x_1=\sum_{\phi=1}^{n_{x_1}} m_\phi x_{1\phi}, \quad x_2=\sum_{\psi=1}^{n_{x_2}} n_\psi x_{2\psi}, \quad \sum_{\phi=1}^{n_{x_1}} m_\phi=1, \quad \sum_{\psi=1}^{n_{x_2}} n_\psi=1 .
		$$
		\\Moreover, since $f(x)=x^t$ is a convex function on $[0, \infty)$, consequently, for any $x_1, x_2 \in \operatorname{co}\left(\operatorname{ext}\left(\mathcal{B}_\mathcal{X}\right)\right)$, we obtain\\
		$$
		\begin{aligned}
			& \|x_1+\tau x_2\|^t+\|\tau x_1-x_2\|^t \\
			&= \left\|\sum_{\phi=1}^{n_{x_1}} m_\phi x_{1\phi}+\tau\left(\sum_{\phi=1}^{n_{x_1}} m_\phi x_2\right)\right\|^t+\left\|\tau\left(\sum_{\phi=1}^{n_{x_1}} m_\phi x_{1\phi}\right)-\left(\sum_{\phi=1}^{n_{x_1}} m_\phi x_2\right)\right\|^t \\
			&=\left\|\sum_{\phi=1}^{n_{x_1}} m_\phi\left(x_{1\phi}+\tau x_2\right)\right\|^t+\left\|\sum_{\phi=1}^{n_{x_1}} m_\phi\left(\tau x_{1\phi}-x_2\right)\right\|^t \\
			&\leq  \left(\sum_{\phi=1}^{n_{x_1}} m_\phi\left\|x_{1\phi}+\tau x_2\right\|\right)^t+\left(\sum_{\phi=1}^{n_{x_1}} m_\phi\left\|\tau x_{1\phi}-x_2\right\|\right)^t \\
			&\leq  \sum_{\phi=1}^{n_{x_1}} m_\phi\left\|x_{1\phi}+\tau x_2\right\|^t+\sum_{\phi=1}^{n_{x_1}} m_\phi\left\|\tau x_{1\phi}-x_2\right\|^t\\
			&=\sum_{\phi=1}^{n_{x_1}} m_\phi\left\|\sum_{\psi=1}^{n_{x_2}} n_\psi x_{1\phi}+\tau\left(\sum_{\psi=1}^{n_{x_2}} n_\psi x_{2\psi}\right)\right\|^t\\
			&+\sum_{\phi=1}^{n_{x_1}} m_\phi\left\|\tau\left(\sum_{\psi=1}^{n_{x_2}} n_\psi x_{1\phi}\right)-\left(\sum_{\psi=1}^{n_{x_2}} n_\psi x_{2\psi}\right)\right\|^t \\
			&=\sum_{\phi=1}^{n_{x_1}} m_\phi\left\|\sum_{\psi=1}^{n_{x_2}} n_\psi\left(x_{1\phi}+\tau x_{2\psi}\right)\right\|^t+\sum_{\phi=1}^{n_{x_1}} m_\phi\left\|\sum_{\psi=1}^{n_{x_2}} n_\psi\left(\tau x_{1\phi}-x_{2\psi}\right)\right\|^t \\
			& \leq \sum_{\phi=1}^{n_{x_1}} m_\phi\left(\sum_{\psi=1}^{n_{x_2}} n_\psi\left\|x_{1\phi}+\tau x_{2\psi}\right\|\right)^t+\sum_{\phi=1}^{n_{x_1}} m_\phi\left(\sum_{\psi=1}^{n_{x_2}} n_\psi\left\|\tau x_{1\phi}-x_{2\psi}\right\|\right)^t\\
			& \leq \sum_{\phi=1}^{n_{x_1}} m_\phi\sum_{\psi=1}^{n_{x_2}} n_\psi\left\|x_{1\phi}+\tau x_{2\psi}\right\|^t+ \sum_{\phi=1}^{n_{x_1}} m_\phi\sum_{\psi=1}^{n_{x_2}} n_\psi\left\|\tau x_{1\phi}-x_{2\psi}\right\|^t\\
			&= \sum_{\phi=1}^{n_{x_1}} m_\phi\sum_{\psi=1}^{n_{x_2}} n_\psi\left(\left\|x_{1\phi}+\tau x_{2\psi}\right\|^t+\left\|\tau x_{1\phi}-x_{2\psi}\right\|^t\right)\\
			&\leq\max\{\left\|x_{1\phi}+\tau x_{2\psi}\right\|^t+\left\|\tau x_{1\phi}-x_{2\psi}\right\|^t:\phi=1,2,...,n_{x_1},\psi=1,2,...,n_{x_2}\}\\
			&\leq\sup\{\left\|x_1+\tau x_2\right\|^t+\left\|\tau x_1-x_2\right\|^t:x_1,x_2\in \operatorname{ext}\left(\mathcal{B}_\mathcal{X}\right)\}.
		\end{aligned}
		$$
		\\Hence, we have\\
		$$\mathcal{J}_t [\tau, \mathcal{X}]\leq\sup \left\{\left(\frac{\|x_1+\tau x_2\|^t+\|\tau x_1- x_2\|^t}{2}\right)^{\frac{1}{t}}:x_1,x_2\in \operatorname{ext}\left(\mathcal{B}_\mathcal{X}\right)\right\}.$$
		\\In addition, it is obvious that\\
		$$\mathcal{J}_t [\tau, \mathcal{X}]\geq\sup \left\{\left(\frac{\|x_1+\tau x_2\|^t+\|\tau x_1- x_2\|^t}{2}\right)^{\frac{1}{t}}:x_1,x_2\in \operatorname{ext}\left(\mathcal{B}_\mathcal{X}\right)\right\}.$$
		\\(iii) On the one hand, it is quite clear that\\
		$$\mathcal{J}_t [\tau, \mathcal{X}]\geq\sup \left\{\mathcal{J}_t [\tau, \mathcal{X}_{0}]: \mathcal{X}_0 \subset \mathcal{X}, \operatorname{dim}\left(\mathcal{X}_0\right)=2\right\}.$$
		\\On the other hand, by the definition of the constant $\mathcal{J}_t [\tau, \mathcal{X}]$, for any $\varepsilon\geq0$,
		there exist $x_1, x_2 \in \mathcal{S}_\mathcal{X}$ such that\\
		$$\mathcal{J}_t [\tau, \mathcal{X}]-\varepsilon\leq \left(\frac{\|x_1+\tau x_2\|^t+\|\tau x_1- x_2\|^t}{2}\right)^{\frac{1}{t}}.$$
		\\Let $\mathcal{X}_0$ be a two-dimensional subspace of $\mathcal{X}$ that incorporates $x_1$ and $x_2$, we have\\
		$$
		\begin{aligned}\mathcal{J}_t [\tau, \mathcal{X}]-\varepsilon&\leq \left(\frac{\|x_1+\tau x_2\|^t+\|\tau x_1- x_2\|^t}{2}\right)^{\frac{1}{t}}\\
			&\leq \mathcal{J}_t [\tau, \mathcal{X}_0] \leq\sup \left\{\mathcal{J}_t [\tau, \mathcal{X}_{0}]: \mathcal{X}_0 \subset \mathcal{X}, \operatorname{dim}\left(\mathcal{X}_0\right)=2\right\}.\end{aligned}$$
		\\Let $\varepsilon \rightarrow 0$, we can obtain\\
		$$\mathcal{J}_t [\tau, \mathcal{X}]\leq\sup \left\{\mathcal{J}_t [\tau, \mathcal{X}_{0}]: \mathcal{X}_0 \subset \mathcal{X}, \operatorname{dim}\left(\mathcal{X}_0\right)=2\right\}.
		$$
		\\The proof is complete.
	\end{proof}
	
	\begin{Example}
		
		Let $\mathcal{X}$ be the space $\mathbb{R}^2$ with the norm defined by\\
		$$
		\left\|\left(x_1, x_2\right)\right\|=\max \left\{\left|x_1+\frac{1}{\sqrt{3}} x_2\right|,\left|x_1-\frac{1}{\sqrt{3}} x_2\right|, \frac{2}{\sqrt{3}}\left|x_2\right|\right\} .
		$$
		\\Then\\
		$$
		\mathcal{J}_t [\tau, \mathcal{X}]= \begin{cases}\left(\frac{(\tau+1)^t+\tau^t}{2}\right)^{\frac{1}{t}},& \tau \geq 1, \\ \left(\frac{(\tau+1)^t+1}{2}\right)^{\frac{1}{t}}, & 0 \leq \tau\leq1 .\end{cases}\\
		$$
	\end{Example}
	\begin{proof} First of all, notice that the unit ball in this norm is a regular hexagon \cite{11}.
		From the above figure, and by doing some simple and careful calculations, we can get that\\
		$$
		\operatorname{ext}\left(\mathcal{B}_\mathcal{X}\right)=\left\{(1,0),\left(\frac{1}{2}, \frac{\sqrt{3}}{2}\right),\left(-\frac{1}{2}, \frac{\sqrt{3}}{2}\right),(-1,0),\left(-\frac{1}{2},-\frac{\sqrt{3}}{2}\right),\left(\frac{1}{2},-\frac{\sqrt{3}}{2}\right)\right\} .
		$$
		\\Now, on account of finite-dimensional Banach spaces must be reflexive spaces, depending on Proposition \ref{1} (ii), we can get\\
		$$
		\mathcal{J}_t [\tau, \mathcal{X}]=\sup \left\{\left(\frac{\|x_1+\tau x_2\|^t+\|\tau x_1- x_2\|^t}{2}\right)^{\frac{1}{t}}:x_1,x_2\in \operatorname{ext}\left(\mathcal{B}_\mathcal{X}\right)\right\}, t\geq1.
		$$
		\\Therefore, through some simple and meticulous calculations, we are not difficult to get\\
		$$
		\mathcal{J}_t [\tau, \mathcal{X}]= \begin{cases}\left(\frac{(\tau+1)^t+\tau^t}{2}\right)^{\frac{1}{t}},& \tau \geq 1, \\ \left(\frac{(\tau+1)^t+1}{2}\right)^{\frac{1}{t}}, & 0 \leq \tau\leq1 .\end{cases}
		$$
	\end{proof}
	
	\begin{Example} Let $\mathcal{X}$ be the space $\mathbb{R}^2$ with the norm defined by\\	$$
		\left\|\left(x_1, x_2\right)\right\|= \begin{cases}\left\|\left(x_1, x_2\right)\right\|_1 & \left(x_1 x_2 \leq 0\right), \\ \left\|\left(x_1, x_2\right)\right\|_{\infty} & \left(x_1 x_2 \geq 0\right) .\end{cases}
		$$
		\\Then\\
		$$
		\mathcal{J}_t [\tau, \mathcal{X}]= \begin{cases}\left(\frac{(\tau+1)^t+\tau^t}{2}\right)^{\frac{1}{t}},& \tau \geq 1, \\ \left(\frac{(\tau+1)^t+1}{2}\right)^{\frac{1}{t}}, & 0 \leq \tau\leq1 .\end{cases}
		$$
	\end{Example}
	\begin{proof}
		First of all, on account of finite-dimensional Banach spaces must be reflexive spaces, depending on Proposition \ref{1} (ii), we can get\\
		$$
		\mathcal{J}_t [\tau, \mathcal{X}]=\sup \left\{\left(\frac{\|x_1+\tau x_2\|^t+\|\tau x_1- x_2\|^t}{2}\right)^{\frac{1}{t}}:x_1,x_2\in \operatorname{ext}\left(\mathcal{B}_\mathcal{X}\right)\right\}, t\geq1.
		$$
		\\Now, since\\
		$$
		\operatorname{ext}\left(\mathcal{B}_\mathcal{X}\right)=\left\{(1,0),(1,1), (0,1),(-1,0),(-1,-1), (0,-1)\right\} .
		$$
		\\Therefore, through some simple and meticulous calculations, we are not difficult to get\\
		$$
		\mathcal{J}_t [\tau, \mathcal{X}]= \begin{cases}\left(\frac{(\tau+1)^t+\tau^t}{2}\right)^{\frac{1}{t}},& \tau \geq 1, \\ \left(\frac{(\tau+1)^t+1}{2}\right)^{\frac{1}{t}}, & 0 \leq \tau\leq1 .\end{cases}
		$$
	\end{proof}

	\begin{Theorem} Let $t_2 \geq t_1 \geq 1$ and $0 \leq \tau \leq 1$. In that way, for any Banach space $\mathcal{X}$,\\
		$$
		\left(\mathcal{J}_{t_1}[\tau,\mathcal{X}]\right)^{t_2} \leq \left(\mathcal{J}_{t_2}[\tau,\mathcal{X}]\right)^{t_2}\leq \frac{(1+\tau)^{t_2}+\left[2 \left(\mathcal{J}_{t_1}[\tau,\mathcal{X}]\right)^{t_1}-(1+\tau)^{t_1}\right]^{\frac{t_2}{t_1}}}{2} .
		$$
		\\In essence, we have that $\mathcal{J}_{ t_1}[\tau,\mathcal{X}]=1+\tau$ if and only if $\mathcal{J}_{ t_2}[\tau,\mathcal{X}]=1+\tau$.
	\end{Theorem}
	
	\begin{proof} As the function $f(x)=x^{\frac{t_{2}}{t_{1}}}$ is a convex function for $t_{2}\geq t_{1}\geq1$, the first inequality is obvious from\\
		$$(\frac{\|x_1+\tau x_2\|^{t_1}+\|\tau x_1- x_2\|^{t_1}}{2})^{\frac{1}{t_1}}\leq(\frac{\|x_1+\tau x_2\|^{t_2}+\|\tau x_1- x_2\|^{t_2}}{2})^{\frac{1}{t_2}}.$$
		\\In order to prove the second inequality, just prove that for any $x_1, x_2 \in \mathcal{S}_\mathcal{X}$, the following inequality is true\\
		$$\|x_1+\tau x_2\|^{t_2}+\|\tau x_1- x_2\|^{t_2}\leq(1+\tau)^{t_2}+[2\left(\mathcal{J}_{t_{1}}[\tau,\mathcal{X}]\right)^{t_{1}}-(1+\tau)^{t_1}]^{\frac{t_2}{t_1}}.$$
		\\Let $v=\max \{\| x_1+\tau x_2\|,\|\tau x_1- x_2\|\}$, then $1 \leq v \leq 1+\tau$. By having in mind that\\ $$\begin{aligned}\|x_1+\tau x_2\|^{t_{1}}+\|\tau x_1- x_2\|^{t_{1}}&=v^{t_{1}}+\operatorname*{min}\{\|x_1+\tau x_2\|^{t_1},|\tau x_1- x_2\|^{t_1}\}\\
			&\leq2\left(\mathcal{J}_{t_{1}}[\tau,\mathcal{X}]\right)^{t_{1}},\end{aligned}$$
		\\we easily have\\ $$\begin{aligned}\|x_1+\tau x_2\|^{t_2}+\|\tau x_1- x_2\|^{t_2}&=v^{t_2}+\operatorname*{min}\{\|x_1+\tau x_2\|^{t_2},|\tau x_1- x_2\|^{t_2}\}\\
			&\leq v^{t_2}+[2\left(\mathcal{J}_{t_{1}}[\tau,\mathcal{X}]\right)^{t_{1}}-v^{t_1}]^{\frac{t_2}{t_1}}=:h(v).\end{aligned}$$
		\\First, as the function $h(v)$ satisfies\\ $$h(1+\tau)=h([2\left(\mathcal{J}_{t_{1}}[\tau,\mathcal{X}]\right)^{t_{1}}-(1+\tau)^{t_1}]^{\frac{1}{t_1}}),$$
		\\and by\\ $$h^{\prime}(v)=t_2v^{t_1-1}[v^{t_2-t_1}-(2\left(\mathcal{J}_{t_{1}}[\tau,\mathcal{X}]\right)^{t_{1}} -v^{t_1})^{\frac{t_2-t_1}{t_1}}],$$
		\\we have that it is decreasing on $[(2\left(\mathcal{J}_{t_{1}}[\tau,\mathcal{X}]\right)^{t_{1}}-(1+\tau)^{t_1})^{\frac{1}{t_1}},\mathcal{J}_{t_{1}}[\tau,\mathcal{X}]]$ and increasing on $[\mathcal{J}_{t_{1}}[\tau,\mathcal{X}],1+\tau],$ holds if\\ $$v\in[(2\left(\mathcal{J}_{t_{1}}[\tau,\mathcal{X}]\right)^{t_{1}}-(1+\tau)^{t_1})^{\frac{1}{t_1}},1+\tau].$$
		\\On the other hand, if \\$$1\leq v\leq[2\left(\mathcal{J}_{t_{1}}[\tau,\mathcal{X}]\right)^{t_{1}}-(1+\tau)^{t_1}]^{\frac{1}{t_1}},$$ \\then\\
		$$\|x_1+\tau x_2\|^{t_2}+\|\tau x_1- x_2\|^{t_2}\leq(1+\tau)^{t_2}+[2\left(\mathcal{J}_{t_{1}}[\tau,\mathcal{X}]\right)^{t_{1}}-(1+\tau)^{t_1}]^{\frac{t_2}{t_1}}.$$
		\\Thus, we obtain the inequality.
		
	\end{proof}
	
	\begin{Corollary}\label{Proposition11}
		Let $t \geq 1$ and $0 \leq \tau\leq   1$. Then, for any Banach space $\mathcal{X}$,\\
		$$
		\left(\mathcal{J}_1[\tau, \mathcal{X}]\right)^{t} \leq \left(\mathcal{J}_t[\tau, \mathcal{X}]\right)^{t}\leq \frac{(1+\tau)^t+\left[2	\mathcal{J}_{1}[\tau,\mathcal{X}] -(1+\tau)\right]^t}{2}.
		$$
		\\In essence, we have that $\mathcal{J}_1[\tau, \mathcal{X}]=1+\tau$ if and only if $\mathcal{J}_{t}[\tau,\mathcal{X}]=1+\tau$.
	\end{Corollary}
	
	\begin{proof}
		The proof of this corollary is similar to that of Theorem \ref{Proposition11}.
	\end{proof}

	\begin{Remark} According to the corollary \ref{Proposition11}, we can conclude that\\$$
		(\mathcal{J}_{1}[\tau,\mathcal{X}])^t \leq \left(\mathcal{J}_t[\tau, \mathcal{X}]\right)^{t} \leq 2^{t-1}\left(\mathcal{J}_{1}[\tau,\mathcal{X}]\right)^t .
		$$
		\\On the other hand, if $0 \leq t \leq 1$, then\\
		$$
		2^{t-1}(\mathcal{J}_{1}[\tau,\mathcal{X}])^t \leq \left(\mathcal{J}_t[\tau, \mathcal{X}]\right)^{t} \leq (\mathcal{J}_{1}[\tau,\mathcal{X}])^t .
		$$
		\\In fact, It can be derived from the inequality below (for any $x_1, x_2 \in \mathcal{S}_\mathcal{X}$, and $\tau \in[0,1]$ ):\\
		$$\begin{aligned}
			(\| x_1+\tau x_2\|+\|\tau x_1- x_2\|)^{t}&\leq\| x_1+\tau x_2\|^t+\|\tau x_1- x_2\|^t\\
			&\leq 2^{1-t}(\| x_1+\tau x_2\|+\|\tau x_1- x_2\|)^{t}.
		\end{aligned}$$\\
	\end{Remark}
	Next, we will compare the constant $\mathcal{J}_t [\tau, \mathcal{X}]$ in relation to the modulus of convexity $\delta_\mathcal{X}(\varepsilon)$. Then, we give the following definition:
	\begin{Definition}
		\cite{10} $\forall\varepsilon\in[0,2],$ the modulus of convexity $\delta_{\mathcal{X}}(\epsilon)$ of the Banach space $\mathcal{X}$ is defined to be\\
		$$\begin{aligned}\delta_{\mathcal{X}}(\epsilon)&=\inf\left\{1-\frac{\|x_1+x_2\|}{2};x_1,x_2\in \mathcal{S}_{\mathcal{X}},\|x_1-x_2\|\geq\epsilon\right\}\\&=\inf\left\{1-\frac{\|x_1+x_2\|}{2};x_1,x_2\in \mathcal{S}_{\mathcal{X}},\|x_1-x_2\|=\epsilon\right\}\\&=\inf\left\{1-\frac{\|x_1+x_2\|}{2};x_1,x_2\in \mathcal{B}_{\mathcal{X}},\|x_1-x_2\|\geq\epsilon\right\}.\end{aligned}$$
	\end{Definition}

	\begin{Proposition}
		Let $t\in \mathbb{R}$, then for any Banach space $\mathcal{X}$, \\
		$$
		\mathcal{J}_t [\tau, \mathcal{X}]\leq\sup \left\{\mathcal{M}_t(3-2\delta_{\mathcal{X}}(\varepsilon)-\tau,\tau\varepsilon+1-\tau):\varepsilon \in[0,2] \right\}.
		$$
	\end{Proposition}
	
	\begin{proof}
		Let $\varepsilon \in[0,2]$, then there exist $x_1,x_2\in \mathcal{S}_\mathcal{X}$ such as $\|x_1-x_2\|=\varepsilon$ and $
		\|x_1+x_2\|\leq2(1-\delta_{\mathcal{X}}(\varepsilon))$.
		
		Since $\|x_1+\tau x_2\|\leq\|x_1+x_2\|+(1-\tau)$ and $\|\tau x_1-x_2\|\leq\tau\|x_1-x_2\|+(1-\tau)$,
		therefore\\ 
		$$\mathcal{M}_ t(\|x_1+\tau x_2\|,\|\tau x_1-x_2\|)\leq \mathcal{M}_t(3-2\delta_{\mathcal{X}}(\varepsilon)-\tau,\tau\varepsilon+1-\tau),
		$$
		\\as required.
	\end{proof}
	
	\begin{Theorem}
		Let $\mathcal{X}$ be uniformly non-square. Then\\
		$$\left(\mathcal{J}_t[\tau,\mathcal{X}]\right)^{t}\leq\operatorname*{max}_{\mathcal{J}(\mathcal{X})\leq\varepsilon\leq2}\frac{[1-\tau+\tau\varepsilon]^{t}+[1+\tau-2\tau\delta_{\mathcal{X}}(\varepsilon)]^{t}}{2}.$$
	\end{Theorem}
	
	\begin{proof}
		For any $x_1,x_2\in \mathcal{S}_{\mathcal{X}}$ and $0\leq\tau\leq1,$ we have\\
		$$\| x_1+\tau x_2\|\leq \tau\|x_1+x_2\|+(1-\tau),$$
		\\and\\$$\|\tau x_1-x_2\|\leq\tau\|x_1-x_2\|+(1-\tau).$$
		\\So, we can conclude\\
		$$\frac{\|x_1+\tau x_2\|^t+\|\tau x_1- x_2\|^t}{2}\leq\frac{[\tau\|x_1+x_2\|+1-\tau]^t+[\tau\|x_1-x_2\|+1-\tau]^t}{2}.$$
		\\By $\mathcal{X}$ is a uniformly non-square space, it is feasible to deduce that $\mathcal{J}(\mathcal{X})<2.$ For $x_1,x_2\in \mathcal{S}_\mathcal{X}$, if we have\\ 
		$$\max\{\|x_1+x_2\|,\|x_1-x_2\|\}\leq \mathcal{J}(\mathcal{X}),$$  \\consequently, this implies that\\
		$$\frac{\|x_1+\tau x_2\|^t+\|\tau x_1- x_2\|^t}{2}\leq[1-\tau+\tau \mathcal{J}(\mathcal{X})]^t.$$
		\\Currently, let us postulate that $\max\{\|x_1+x_2\|,\|x_1-x_2\|\}\geq \mathcal{J}(\mathcal{X}),$ we may assume that $\varepsilon=:\|x_1-x_2\|\geq \mathcal{J}(\mathcal{X}).$ Then by applying $\|x_1+x_2\|\leq2(1-\delta_\mathcal{X}(\varepsilon))$, we have\\
		$$\frac{\|x_1+\tau x_2\|^t+\|\tau x_1- x_2\|^t}{2}\leq\max_{\mathcal{J}(\mathcal{X})\leq\varepsilon\leq2}\frac{[1-\tau+\tau\varepsilon]^t+[1+\tau-2\tau\delta_\mathcal{X}(\varepsilon)]^t}{2}.$$  
		\\On the other hand, $\delta_\mathcal{X}(\mathcal{J}(\mathcal{X}))=1-\frac{\mathcal{J}(\mathcal{X})}{2}$ is valid for $\mathcal{J}(\mathcal{X})<2,$ hence\\
		$$[1-\tau+\tau \mathcal{J}(\mathcal{X})]^t\leq\max_{\mathcal{J}(\mathcal{X})\leq\varepsilon\leq2}\frac{[1-\tau+\tau\varepsilon]^t+[1+\tau-2\tau\delta_\mathcal{X}(\varepsilon)]^t}{2}.$$
		\\Therefore, this completes the proof.
	\end{proof}

	Last but not least, we use this skew constant $\mathcal{J}_t [\tau, \mathcal{X}]$ to define a new constant $\mathcal{G}_t(\mathcal{X})$.
	\begin{Definition}
		
		Let $\tau\geq0$ and $-\infty \leq t<\infty$, \\
		$$
		\mathcal{G}_t(\mathcal{X})=\sup \left\{\frac{\left(\mathcal{J}_t[\tau, \mathcal{X}]\right)^2}{1+\tau^2}: x_1,x_2\in \mathcal{S}_\mathcal{X}\right\}.
		$$
	\end{Definition}
	
	In particular, taking $t=-\infty$ yields the constant $\mathcal{G}_{-\infty}(\mathcal{X})$.\\
	We use the James constant $\mathcal{J} (\mathcal{X})$ to give an estimate of the constant $\mathcal{G}_{-\infty}(\mathcal{X})$. For convenience, the James constant $\mathcal{J} (\mathcal{X})$ is simply $\mathcal{J}$.

	\begin{Theorem}
		Let $\mathcal{X}$ be a Banach space, then\\
		$$\mathcal{G}_{-\infty}(\mathcal{X})\leq (\mathcal{J}-1)^2+\frac{4(\mathcal{J}-1)^2a}{(\mathcal{J}^2-2\mathcal{J}+a)^2+4(\mathcal{J}-1)^2},$$
		\\where $a=\sqrt{(2\mathcal{J}-\mathcal{J}^2)^2+4(\mathcal{J}-1)^2}$.\\
	\end{Theorem}
	\begin{proof}
		For any $x_1,x_2 \in \mathcal{S}_\mathcal{X}\text{ and }0\leq\tau\leq1$, we have\\
		$$\| x_1+\tau x_2\|\leq \tau\|x_1+x_2\|+(1-\tau),$$
		\\and\\$$\|\tau x_1-x_2\|\leq\tau\|x_1-x_2\|+(1-\tau).$$
		\\Using the definition of $\mathcal{J}(\mathcal{X})$, the following relation is obtained:\\
		$$\min\{\| x_1+\tau x_2\|^2,\|\tau x_1- x_2\|^2\}\leq\tau^2\mathcal{J}^2+2\tau(1-\tau)\mathcal{J}+(1-\tau)^2.$$
		\\From the definition of the constant $\mathcal{G}_{-\infty}(\mathcal{X})$ and the relation above, we have\\
		$$\begin{aligned}
			\mathcal{G}_{-\infty}(\mathcal{X})&=\sup\biggl\{\frac{\min\{\| x_1+\tau x_2\|^2,\|\tau x_1- x_2\|^2\}}{1+\tau^2}:x_1,x_2\in \mathcal{S}_\mathcal{X}\biggr\}\\
			&\leq \max_{0\leq\tau\leq1}\frac{\tau^2\mathcal{J}^2+2\tau(1-\tau)\mathcal{J}+(1-\tau)^2}{1+\tau^2}.
		\end{aligned}$$
		\\Now, we only need to consider the maximum value of the function\\ $$\mathcal{G}(\tau)=\frac{\tau^{2}\mathcal{J}^{2}+2(1-\tau)\tau \mathcal{J}+(1-\tau)^{2}}{1+\tau^{2}}$$ \\in $\tau\in[0,1]$. By taking the derivative to find the extreme point, we can see that $\mathcal{G}(\tau)$ reaches its maximum at $\tau=\frac{a+(\mathcal{J}^{2}-2\mathcal{J})}{2(\mathcal{J}-1)}$, here\\ $$a=\sqrt{(2\mathcal{J}-\mathcal{J}^{2})^{2}+4(\mathcal{J}-1)^{2}}.$$ \\Notice that $2(\mathcal{J}-1)\tau=a+(\mathcal{J}^2-2\mathcal{J}),$ so\\
		$$\begin{aligned}&\max_{0\leq\tau\leq1}\frac{\tau^2\mathcal{J}^2+2(1-\tau)\tau \mathcal{J}+(1-\tau)^2}{1+\tau^2}\\
			&=(\mathcal{J}-1)^2+\frac{4(\mathcal{J}-1)^2a}{(\mathcal{J}^2-2\mathcal{J}+a)^2+4(\mathcal{J}-1)^2}.\end{aligned}$$
		\\Thus, the estimate in the theorem is proved.	
	\end{proof}

	\begin{Example} 
		
		We can use Theorem 3.3 to estimate the range of values of $\mathcal{G}_{-\infty}(\mathcal{X})$ on the Day-James space $l_2-l_1$.\\
		Given that $\mathcal{J}(l_2-l_1)=\sqrt{\frac{8}{3}}$, use the estimate from Theorem 3.3\\
		$$\mathcal{G}_{-\infty}(\mathcal{X})\leq (\mathcal{J}-1)^2+\frac{4(\mathcal{J}-1)^2a}{(\mathcal{J}^2-2\mathcal{J}+a)^2+4(\mathcal{J}-1)^2},$$
		\\It can be obtained\\
		$$\begin{aligned}
			&\mathcal{G}_{-\infty}(l_2-l_1)\\
			&\leq(\mathcal{J}-1)^2+\frac{4(\mathcal{J}-1)^2a}{(\mathcal{J}^2-2\mathcal{J}+a)^2+4(\mathcal{J}-1)^2}\\
			&=(\mathcal{J}-1)^2+\frac{4(\mathcal{J}-1)^2\sqrt{(2\mathcal{J}-\mathcal{J}^2)^2+4(\mathcal{J}-1)^2}}{(\mathcal{J}^2-2\mathcal{J}+\sqrt{(2\mathcal{J}-\mathcal{J}^2)^2+4(\mathcal{J}-1)^2})^2+4(\mathcal{J}-1)^2}\\
		\end{aligned}$$
		$$\begin{aligned}
			&=(\sqrt{\frac{8}{3}}-1)^2\\
			&+\frac{4(\sqrt{\frac{8}{3}}-1)^2\sqrt{(2\sqrt{\frac{8}{3}}-(\sqrt{\frac{8}{3}})^2)^2+4(\sqrt{\frac{8}{3}}-1)^2}}{((\sqrt{\frac{8}{3}})^2-2\sqrt{\frac{8}{3}}+\sqrt{(2\sqrt{\frac{8}{3}}-(\sqrt{\frac{8}{3}})^2)^2+4(\sqrt{\frac{8}{3}}-1)^2})^2+4(\sqrt{\frac{8}{3}}-1)^2}\\
			&\approx1.4007<\sqrt{2}=\mathcal{C}_\mathcal{Z}(l_2-l_1).
		\end{aligned}$$
		\\Therefore, the Day-James space $l_2-l_1$ is an example of $\mathcal{G}_{-\infty}(\mathcal{X})<\mathcal{C}_{\mathbf{\mathcal{Z}}}(\mathcal{X})$.

	\end{Example}

	
	\subsection*{Acknowledgments}
	This work was completed with the support of Anhui Province Higher Education Science
	Research Project(Natural Science), 2023AH050487.

\end{document}